\documentclass[12pt]{article}

\usepackage{geometry} 
\usepackage{amsthm,amssymb,enumerate}
\usepackage{amsmath,amsfonts,latexsym}
\usepackage{latexsym}
\usepackage{float}
\usepackage{xcolor}
\usepackage{url}

\usepackage{mathrsfs}
\usepackage{hyperref}

\DeclareMathOperator{\Aut}{Aut}

\DeclareMathOperator{\PSL}{PSL}

\DeclareMathOperator{\PGL}{PGL}

\DeclareMathOperator{\ag}{AG}
\DeclareMathOperator{\pg}{PG}

\DeclareMathOperator{\PGaL}{P\Gamma L}

\DeclareMathOperator{\AGL}{AGL}

\DeclareMathOperator{\Sym}{Sym}

\DeclareMathOperator{\alt}{A}
\DeclareMathOperator{\s}{S}
\DeclareMathOperator{\mg}{M}

\renewcommand{\leq}{\leqslant}
\renewcommand{\geq}{\geqslant}
\newcommand{\ra}{\rightarrow}

\newcommand{\F}{\mathbb F}

\newcommand{\G}{\mathcal{G}}

\newcommand{\B}{\mathcal B}

\newcommand{\T}{\mathcal T}

\newcommand{\Z}{\mathbb Z}

\theoremstyle{plain}
\newtheorem{lemma}{Lemma}
\newtheorem{theorem}[lemma]{Theorem}

\theoremstyle{definition}
\newtheorem{definition}[lemma]{Definition}
\newtheorem{example}[lemma]{Example}

\newtheorem{problem}[lemma]{Problem}

\numberwithin{equation}{section}
\numberwithin{lemma}{section}

\begin{document}

%opening
\title{Neighbour-transitive codes in {K}neser graphs\thanks{This work has been supported by the Croatian Science Foundation under the project 5713.\\ 
{\bf Key words:} neighbour-transitive code, completely transitive code, Kneser graph\\
{\bf 2020 Mathematics subject classification:} 05E18,20B25}}
\author{Dean Crnkovi{\'c}, Daniel R. Hawtin, Nina Mostarac and Andrea {\v S}vob}
\date{
 \small{
  \emph{
   Faculty of Mathematics, University of Rijeka\\
   51000 Rijeka, Croatia\\
  }
  \vspace{0.25cm}
  \href{mailto:deanc@math.uniri.hr}{deanc@math.uniri.hr}\hspace{0.5cm} \href{mailto:dhawtin@math.uniri.hr}{dhawtin@math.uniri.hr}\\ 
  \href{mailto:nmavrovic@math.uniri.hr}{nmavrovic@math.uniri.hr}\hspace{0.5cm} \href{mailto:asvob@math.uniri.hr}{asvob@math.uniri.hr}\\
  \vspace{0.25cm}
 }
\today
}

\maketitle

\begin{abstract}
 A \emph{code} $C$ is a subset of the vertex set of a graph and $C$ is \emph{$s$-neighbour-transitive} if its automorphism group ${\rm Aut}(C)$ acts transitively on each of the first $s+1$ parts $C_0,C_1,\ldots,C_s$ of the \emph{distance partition} $\{C=C_0,C_1,\ldots,C_\rho\}$, where $\rho$ is the \emph{covering radius} of $C$. While codes have traditionally been studied in the Hamming and Johnson graphs, we consider here codes in the Kneser graphs. Let $\Omega$ be the underlying set on which the Kneser graph $K(n,k)$ is defined. Our first main result says that if $C$ is a $2$-neighbour-transitive code in $K(n,k)$ such that $C$ has minimum distance at least $5$, then $n=2k+1$ ({\em i.e.,} $C$ is a code in an odd graph) and $C$ lies in a particular infinite family or is one particular sporadic example. We then prove several results when $C$ is a neighbour-transitive code in the Kneser graph $K(n,k)$. First, if ${\rm Aut}(C)$ acts intransitively on $\Omega$ we characterise $C$ in terms of certain parameters. We then assume that ${\rm Aut}(C)$ acts transitively on $\Omega$, first proving that if $C$ has minimum distance at least $3$ then either $K(n,k)$ is an odd graph or ${\rm Aut}(C)$ has a $2$-homogeneous (and hence primitive) action on $\Omega$. We then assume that $C$ is a code in an odd graph and ${\rm Aut}(C)$ acts imprimitively on $\Omega$ and characterise $C$ in terms of certain parameters. We give examples in each of these cases and pose several open problems.
\end{abstract}

\section{Introduction}

Let $\varGamma$ be a simple connected graph on finitely many vertices. For vertices $\alpha$ and $\beta$ of $\varGamma$ denote the graph distance between $\alpha$ and $\beta$ by $d(\alpha,\beta)$. A \emph{code} $C$ in $\varGamma$ is a subset of the vertex set of $\varGamma$, the elements of which are called \emph{codewords}. We assume throughout that $|C|\geq 2$. The \emph{minimum distance} $\delta$ of $C$ is the smallest distance, in the graph $\varGamma$, between two distinct codewords of $C$. The \emph{set of $i$-neighbours} of $C$, denoted $C_i$, is the subset of all those vertices $\gamma$ of $\varGamma$ for which there exists $\alpha\in C$ such that $d(\alpha,\gamma)=i$ and $d(\beta,\gamma)\geq i$, for all $\beta\in C$. The \emph{distance partition} of $C$ is then the partition $\{C,C_1,C_2,\ldots,C_\rho\}$ of the vertex set of $\varGamma$, where $\rho$ is the \emph{covering radius} of $C$, that is, the maximum value of $i$ for which $C_i$ is non-empty. The automorphism group $\Aut(C)$ of $C$ is the setwise stabiliser of $C$ inside $\Aut(\varGamma)$. We consider the following symmetry conditions:

\begin{definition}\label{sneighbourtransdef}
 Let $C$ be a code with covering radius $\rho$ in a graph $\varGamma$, let $G\leq\Aut(C)$ and let $i\in\{1,\ldots,\rho\}$. Then $C$ is said to be
 \begin{enumerate}[(1)]
  \item \emph{$(G,i)$-neighbour-transitive} if $G$ acts transitively on each of the sets $C,C_1,\ldots, C_i$,
  \item \emph{$G$-neighbour-transitive} if $C$ is $(G,1)$-neighbour-transitive, and,
  \item \emph{$G$-completely transitive} if $C$ is $(G,\rho)$-neighbour-transitive.
 \end{enumerate}
 If $G=\Aut(C)$ we simply say $C$ is, in each respective part, (1) \emph{$i$-neighbour-transitive}, (2) \emph{neighbour-transitive} and (3) \emph{completely transitive}.
\end{definition}

The study of codes in general graphs was proposed in 1973 \cite{Biggs1973289,delsarte1973algebraic} and has gained interest recently, for example, in \cite{huang2018perfect,krotov2016perfect}. Neighbour-transitive codes have been investigated in Hamming graphs \cite{gillespieCharNT}, where a characterisation of one class of these codes is obtained. Stronger results were obtained in \cite{minimal2nt} for binary (that is, alphabet size two) Hamming graphs for $2$-neighbour-transitive codes, where all such codes with minimum distance $5$ are characterised. Completely transitive codes are a subclass of completely regular codes (see \cite{borges2019completely}) and completely transitive codes in Hamming graphs have been studied, for instance, in \cite{borges2001nonexistence,Giudici1999647}. In \cite{bailey2020classification}, the characterisation \cite{minimal2nt} is applied in order to classify binary completely transitive codes in Hamming graphs with certain classes of automorphism group and having minimum distance at least $5$. Neighbour-transitive codes in Johnson graphs were considered by Praeger {\em et al.}~\cite{bamberg2022codes,liebler2014neighbour}, with their results nearing a full classification. Several results have been obtained for neighbour-transitive codes in the incidence graphs of generalised quadrangles \cite{NTcodesGQs}, where such codes were found to have interesting geometric properties. %In the Grassmann graphs, the so called `$q$-analogues' of the Johnson graphs, neighbour-transitive codes are studied in \cite{NTcodesGrassmann}. 

Let $\Omega$ be the underlying set of the Kneser graph $\varGamma=K(n,k)$, so that $|\Omega|=n$, the vertices of $\varGamma$ are ${\Omega \choose k}$ and two vertices $\alpha,\beta\in V\varGamma$ are defined to be adjacent if $\alpha$ and $\beta$ are disjoint. In order that $\varGamma$ is not a complete graph and has valency at least $2$, we assume throughout that $2\leq k\leq (n-1)/2$. If $n=2k+1$ then $\varGamma$ is the \emph{odd graph} $O_{k+1}$. Our first main result concerns $2$-neighbour-transitive codes in $\varGamma$. Note that, for $n=23$, the \emph{endecads} are the subsets of $\Omega$ with characteristic vectors corresponding to the weight $11$ codewords of the perfect binary Golay code in the Hamming graph (see \cite[Page 71]{conway1985atlas}).

\begin{theorem}\label{Kneser2NTdelta5}
 Let $C$ be a $2$-neighbour-transitive code in $\varGamma=K(n,k)$ with minimum distance $\delta\geq 5$. Then $n=2k+1$, and hence $\varGamma$ is the odd graph $O_{k+1}$, and one of the following holds.
 \begin{enumerate}[{\rm (1)}]
  \item $\Aut(C)\cong\mg_{23}$ with $n=23$ and $C$ consists of the endecads.
  \item $\Aut(C)\cong\PGaL_d(2)$, where $d\geq 5$, $\Omega$ is the set of all points, and $C$ is the set of all hyperplanes, in $\pg_{d-1}(2)$.
 \end{enumerate}
\end{theorem}

Given the result above, we next consider weaker condition of neighbour-transitivity for codes in Kneser graphs. Let $C$ be a neighbour-transitive code in $K(n,k)$. It is common to analyse group actions successively by whether they are (i) intransitive, (ii) transitive but imprimitive, and (iii) primitive. We roughly follow this approach concerning the induced action of $\Aut(C)$ on $\Omega$. For the intransitive case, we have the following result. Note that we call a code $C$ \emph{trivial} if $|C|=1$.

\begin{theorem}\label{intransitiveParamThm}
 Let $C$ be a non-trivial neighbour-transitive code in $K(n,k)$ with minimum distance $\delta$ and suppose that $\Aut(C)$ acts intransitively on $\Omega$. Then $\Aut(C)$ has precisely two non-empty orbits on $\Omega$, say $U$ and $V$, and one of the following holds (up to interchanging $U$ and $V$):
 \begin{enumerate}[{\rm (1)}]
  \item $\delta=1$, $|U|=1$ and $\alpha\subseteq V$ for all $\alpha\in C$.
  \item $\delta=1$, $n=2k+1$, $|U|$ is even and $|\alpha\cap U|=|U|/2$ for all $\alpha\in C$.
  \item $\delta\geq 2$, $|U|<k$ and $U\subseteq \alpha$ for all $\alpha\in C$.
 \end{enumerate}
 Moreover, in each case, an infinite family of examples exists (see Lemma~\ref{lem-intransExamples}).
\end{theorem}

If $C$ is a neighbour-transitive code with minimum distance at least $3$ in a Kneser graph and $\Aut(C)$ acts transitively on $\Omega$ then we prove the following. Note that $2$-homogeneous implies primitive.

\begin{theorem}\label{notodd2hom}
 Let $C$ be a neighbour-transitive code with minimum distance $\delta\geq 3$ in $K(n,k)$ and suppose that $\Aut(C)$ acts transitively on $\Omega$. Then either $K(n,k)$ is an odd graph or $\Aut(C)$ acts $2$-homogeneously on $\Omega$.
\end{theorem}

Since $2$-homogeneous groups have been classified (see \cite[Theorem 9.4B and Section 7.7]{dixon1996permutation}), we pose the following problem.

\begin{problem}
 Classify neighbour-transitive codes $C$ in Kneser graphs with minimum distance $\delta\geq 3$ and where $\Aut(C)$ acts $2$-homogeneously on $\Omega$.
\end{problem}

Given the above result, we turn our attention to codes in odd graphs with a transitive but imprimitive action on $\Omega$. Note that in Theorem~\ref{imprimitiveOddThm}, the fact that $ab=2k+1$ implies that $a$ and $b$ are both odd. Also, we use the notation $\{b_1^{a_1},\ldots,b_s^{a_s}\}$ for a multiset containing elements $b_i$ with multiplicity $a_i$, for $i=1,\ldots,s$.

\begin{theorem}\label{imprimitiveOddThm}
 Let $C$ be a non-trivial neighbour-transitive code in $K(2k+1,k)=O_{k+1}$ with minimum distance $\delta$ and suppose that $\Aut(C)$ acts transitively but imprimitively on $\Omega$ with system of imprimitivity $\{B_1,\ldots,B_a\}$, where $|B_i|=b$ for each $i$ and $ab=2k+1$ with $a,b\geq 2$. Moreover, for each $\alpha\subseteq \Omega$ let $\iota(\alpha)$ be the multiset $\{|\alpha\cap B_i|\mid i=1,\ldots,a\}$. Then one of the following holds.
 \begin{enumerate}[{\rm (1)}]
  \item $\delta\geq 2$ and there exist integers $a_0,a_1$ and $b_1$ such that $\iota(\alpha)=\{b^{a_0},b_1^{a_1}\}$ for all $\alpha\in C$.
  \item $\delta=1$ and $\iota(\alpha)=\left\{\left(\frac{b-1}{2}\right)^{(a+1)/2},\left(\frac{b+1}{2}\right)^{(a-1)/2}\right\}$ for all $\alpha\in C$.
  \item $\delta=1$ and $\iota(\alpha)=\left\{0^{(a-1)/2},\frac{b-1}{2},b^{(a-1)/2}\right\}$ for all $\alpha\in C$.
 \end{enumerate}
 In each case an infinite family of examples exists (see Lemma~\ref{lem-imprimitiveExamples}).
\end{theorem}
% 
% Note that the above theorem concerns only odd graphs. Thus we pose the following problem.
% 
% \begin{problem}
%  Prove a similar result to Theorem~\ref{imprimitiveOddThm} for the Kneser graph $K(n,k)$ when $n\neq 2k+1$.
% \end{problem}

Theorems~\ref{intransitiveParamThm} and~\ref{imprimitiveOddThm} characterise neighbour-transitive codes in odd graphs where the action on $\Omega$ is intransitive or imprimitive. The following problem concerns the case where the action on $\Omega$ is primitive.

\begin{problem}
 Given a primitive group of O'Nan--Scott type $X$, do there exist neighbour-transitive codes $C$ in odd graphs where $\Aut(C)^\Omega$ has type $X$ and $C$ has minimum distance at least $3$?
\end{problem}

Note that Example~\ref{halfLinesinPlaneEx} is an example of a neighbour-transitive code $C$ where $\Aut(C)$ is an almost-simple group, but that in this case $C$ has minimum distance $1$.

In some sense, the codes in Lemmas~\ref{lem-intransExamples} and \ref{lem-imprimitiveExamples} arise in a straightforward manner, being the sets of all those vertices satisfying a given part of Theorem~\ref{intransitiveParamThm} or Theorem~\ref{imprimitiveOddThm}. Example~\ref{ag32tetrahedronEx} shows that these are not the only codes satisfying Theorem~\ref{intransitiveParamThm}. The following asks if this is the case for Theorem~\ref{imprimitiveOddThm}.

\begin{problem}
 Do there exist codes satisfying Theorem~\ref{imprimitiveOddThm} other than those given by Lemma~\ref{lem-imprimitiveExamples}?
\end{problem}

% 
% If $C$ is a neighbour-transitive code in an odd graph then either $C$ satisfies Theorem~\ref{intransitiveParamThm}, Theorem~\ref{imprimitiveOddThm} or $\Aut(C)$ acts primitively on $\Omega$. Moreover, Lemma~\ref{homogeneouslemma} shows that if $C$ has minimum distance $\delta\geq 3$ then $\Aut(C)$ acts transitively on non-incident codeword-point pairs (or, in the language of design theory, $C$ is \emph{antiflag-transitive}). We give an example of such a code in . Moreover, we pose the following.

This paper is organised as follows. Section~\ref{prelims} introduces notation and preliminary results used throughout. Section~\ref{2NTsection} concerns $2$-neighbour-transitive codes, proving Theorem~\ref{Kneser2NTdelta5}. We then turn our attention to neighbour-transitive codes. Section~\ref{intransitiveSection} addresses the `intransitive' case, providing the examples for, and proving, Theorem~\ref{intransitiveParamThm}. Section~\ref{imprimitiveSection} includes various examples for the `imprimitive' case and proves Theorem~\ref{imprimitiveOddThm}. Finally, in Section~\ref{transitiveSection}, we prove Theorem~\ref{notodd2hom}.

\section{Preliminaries}\label{prelims}

%\textcolor{blue}{DH -- Probably need to add a bunch of stuff to this section. (In particular, basic stuff about group actions. What else?)}

Throughout this paper we consider the Kneser graph $\varGamma=K(n,k)$ with underlying set $\Omega$ of size $n$, so that $V\varGamma={\Omega \choose k}$ and vertices $\alpha,\beta\in V\varGamma$ are adjacent if $\alpha$ and $\beta$ are disjoint. Recall that we assume $2\leq k\leq \frac{n-1}{2}$. For a vertex $\alpha\in\varGamma$, the set of $i$-neighbours of $\alpha$ is denoted by $\varGamma_i(\alpha)=\{ \beta\in V\varGamma \mid d(\alpha,\beta)=i \}$. For a subset $U\subseteq\Omega$, we denote by $\overline{U}$ the complement $\Omega\setminus U$. The automorphism group $\Aut(\varGamma)$ of the Kneser graph is the symmetric group $\Sym(\Omega)\cong\s_n$ \cite[Corollary~7.8.2]{godsil2013algebraic}. If a group $G$ acts transitively on a set $U$ then we say that $U$ is a $G$-orbit. For more information about permutation groups see \cite{dixon1996permutation}.

%Let $\alpha$ and $\beta$ be vertices of $K(n,k)$ such that $d(\alpha,\beta)=2$. Then, since $\alpha\cup\beta$ must be disjoint from some vertex $\beta$, we have that $1\leq |\alpha\cap\beta|\leq $, but have size at most . If $k$ is odd, then $|\alpha\cap \beta|=(k-1)/2$, and if $k$ is even, then $|\alpha\cap \beta|=n-k/2-1$.

\begin{definition}\label{GinvDef}
 Let $G$ be a group acting on a set $\Omega$. A \emph{$G$-invariant map} on $\Omega$ is a map $\iota:\Omega\ra S$, for some set $S$, such that for all $\alpha\in\Omega$ and $g\in G$ it holds that $\iota(\alpha)=\iota(\alpha^g)$. If $\iota$ is a $G$-invariant map on $\Omega$ and $\alpha\in\Omega$ then we say that $\alpha$ has \emph{type} $\iota(\alpha)$.
\end{definition}

A $G$-invariant map can be a simple way to rule out the existence of a neighbour-transitive code with automorphism group contained in a prescribed group $G$. In particular, producing counterexamples to the following lemma will be one important tool in proving our main results. %Note that while Lemmas~\ref{invlemma} and~\ref{invlemmadelta} are stated in more generality we only require the $s=1$ case here. 

%\textcolor{blue}{DH -- we should work out what we actually need from the two lemmas below, and we'll need to include some proof.}

\begin{lemma}\label{invlemma}
 Let $C$ be a neighbour-transitive code with minimum distance $\delta$ in the graph $\varGamma$, let $G\leq \Aut(\varGamma)$ such that $\Aut(C)\leq G$, and let $\iota$ be a $G$-invariant map on the vertices of $\varGamma$. Then the following hold:
 \begin{enumerate}[{\rm (1)}]
  \item If $i\in\{0,1\}$ then all vertices in $C_i$ have the same type.
  \item If $\alpha\in C$ and $\beta\in C_1$ then every vertex in $\varGamma_1(\alpha)$ has the same type as $\alpha$ or $\beta$.
  \item If $\delta\geq 2$ and $\alpha\in C$ then every pair of vertices in $\varGamma_1(\alpha)$ have the same type. 
 \end{enumerate}
\end{lemma}

\begin{proof}
 Since $C(=C_0)$ and $C_1$ are $\Aut(C)$-orbits, and $\Aut(C)\leq G$ preserves the type of a vertex, part (1) holds. Suppose $\alpha\in C$, $\beta\in C_1$ and $\gamma\in\varGamma_1(\alpha)$. Then $\gamma$ is either in $C$ or in $C_1$ and so, by part (1), has the same type as $\alpha$ or $\beta$, proving part (2). If $\delta\geq 2$ then $\gamma\notin C$ and part (3) holds.
\end{proof}

% 
% \begin{lemma}\cite[Lemma~1.3]{NTcodesGrassmann}\label{invlemmadelta}
%  Let $C$ be an $s$-neighbour-transitive code with minimum distance $\delta\geq 2s+1$ in the graph $\varGamma$, let $\alpha\in C$, let $G\leq \Aut(\varGamma)$ such that $\Aut(C)_\alpha\leq G$, and let $\iota$ be a $G$-invariant map on the vertices of $\varGamma$. Then all pairs of vertices in $\varGamma_s(\alpha)$ have the same type.
% \end{lemma}

\section{\texorpdfstring{$2$}{2}-Neighbour-transitive codes}\label{2NTsection}

First we prove that there are no $2$-neighbour-transitive codes with minimum distance at least $5$ in a Kneser graph, unless it is in fact an odd graph.

\begin{theorem}\label{no2ntdelta5kneser}
 There are no $2$-neighbour-transitive codes with minimum distance $\delta\geq 5$ in the Kneser graph $\varGamma=K(n,k)$ when $k\neq (n-1)/2$, that is, when $\varGamma$ is not an odd graph.
\end{theorem}

\begin{proof}
 Suppose $C$ is a $2$-neighbour-transitive code with minimum distance $\delta\geq 5$ in $\varGamma$. Then any pair of balls of radius $2$ centred at distinct codewords are disjoint. In particular, since $\Aut(C)$ acts transitively on $C_1$, if $\alpha\in C$ then $\Aut(C)_\alpha$ must act transitively on $\varGamma_2(\alpha)$. Since $\varGamma$ is vertex-transitive, we may assume that $\alpha=\{1,2,\ldots,k\}$. Consider the vertices $\beta_1=\{2,3,\ldots,k+1\}$, $\beta_2=\{3,4,\ldots,k+2\}$ and $\gamma=\{k+3,k+4,\ldots,2k+2\}$. As $|\alpha\cap\beta_1|=k-1$ and $|\alpha\cap\beta_2|=k-2$ and $\gamma$ is disjoint from each of $\alpha$, $\beta_1$ and $\beta_2$, we have that $\beta_1,\beta_2\in\varGamma_2(\alpha)$. However, $\Aut(C)_\alpha\leq \Sym(\alpha)\times\Sym(\overline{\alpha})$ and, since $|\alpha\cap\beta_1|\neq|\alpha\cap\beta_2|$, there is no element of $\Aut(C)_\alpha$ mapping $\beta_1$ to $\beta_2$, which gives a contradiction, and the result holds.
\end{proof}

Let $C$ be a code in the odd graph $O_{k+1}=K(2k+1,k)$. Then, since $C$ consists of a set of $k$-subsets of the set $\Omega$, which has size $2k+1$, we may consider $C$ to also be a code in the Johnson graph $J(2k+1,k)$ by simply taking $\Omega$ to be the underlying subset upon which $J(2k+1,k)$ is defined. Note that $k$-subsets $\alpha,\beta\subseteq \Omega$ are adjacent in $J(2k+1,k)$ precisely when they are at distance $2$ in $O_k$.

\begin{lemma}\label{oddtojohnson}
 Let $C$ be a $2$-neighbour-transitive code in $O_{k+1}$ with minimum distance $\delta\geq 5$. Then $C$ is also a code in the Johnson graph $J(2k+1,k)$, with the same vertex set as $O_{k+1}$, and $C$ is neighbour-transitive in $J(2k+1,k)$ with minimum distance $\delta'\geq 3$. %\emph{(Can we prove the converse?)}
\end{lemma}

\begin{proof}
 First note that the vertex set of $O_{k+1}$ and the vertex set of $J(2k+1,k)$ are both ${\Omega \choose k}$, so that $C$ is indeed a code in $J(2k+1,k)$. For clarity, let $C'$ be the code in $J(2k+1,k)$. Since $k\neq (2k+1)/2$, we have that $\Aut(J(2k+1,k))=\Sym(\Omega)$, and hence $\Aut(J(2k+1,k))=\Aut(O_{k+1})$. In particular, $\Aut(C)=\Aut(C')$ and so $\Aut(C')$ acts transitively on $C'$. Since $C$ is $2$-neighbour-transitive, $\Aut(C)$ acts transitively on $C_2$. In the remainder of the proof, we will show that $\delta'\geq 3$ and that $C_2=C'_1$, which implies that $\Aut(C')$ acts transitively on $C'_1$, from which the result will follow. 
 
 First, if $C'$ has  minimum distance $\delta'=1$ then there would exist $\alpha_1,\alpha_2\in C'$ such that $|\alpha_1\cap\alpha_2|=k-1$, in which case $d_{O_{k+1}}(\alpha_1,\alpha_2)=2$ and $C$ would have minimum distance $\delta\leq 2$. Similarly, if $C'$ has minimum distance $\delta'=2$ then there would exist $\alpha_1,\alpha_2\in C'$ such that $|\alpha_1\cap\alpha_2|=k-2$, in which case $d_{O_{k+1}}(\alpha_1,\alpha_2)=4$ and $C$ would have minimum distance $\delta\leq 4$. Since $\delta\geq 5$ we have that $C'$ has minimum distance $\delta'\geq 3$.
 
 Let $\beta\in C_2$. Then there exists some $\alpha\in C$ such that $d_{O_{k+1}}(\alpha,\beta)=2$, that is, $|\alpha\cap\beta|=k-1$ so that $\alpha$ and $\beta$ are adjacent in $J(2k+1,k)$. Since $\delta'\geq 3$, it follows that $\beta\in C'_1$. This implies that $C_2\subseteq C'_1$. Let $\beta'\in C'_1$. Then there exists an $\alpha'$ in $C'$, and hence in $C$, such that $|\alpha'\cap\beta'|=k-1$. Hence $d_{O_{k+1}}(\alpha',\beta')=2$. This implies, since $\beta'\notin C$ that $\beta'$ is in $C_1$ or $C_2$. Suppose that $\beta'\in C_1$ so that there exists some $\gamma\in C$ such that $d_{O_{k+1}}(\gamma,\beta')=1$. Then $\gamma$ and $\beta'$ are disjoint. This implies that either $\gamma$ and $\alpha'$ are disjoint, in which case $d_{O_{k+1}}(\gamma,\alpha')=1$, or $|\alpha'\cap\gamma|=1$, in which case $d_{O_{k+1}}(\gamma,\alpha')=3$. This gives a contradiction and hence $\beta'\in C_2$ so that $C'_1=C_2$, proving the result.
\end{proof}

We are now ready to prove Theorem~\ref{Kneser2NTdelta5}.

\begin{proof}[Proof of Theorem~\ref{Kneser2NTdelta5}]
 Suppose that $C$ is a $2$-neighbour-transitive code in $O_{k+1}$ with minimum distance $\delta\geq 5$. Then, by Lemma~\ref{oddtojohnson}, $C$ is a neighbour-transitive code in $J(2k+1,k)$ with minimum distance $\delta'\geq 3$. Thus, we may apply the results of \cite{liebler2014neighbour} and \cite{neunhoffer2014sporadic} in order to prove the result. Note that since $|\Omega|$ is odd, we can immediately rule out any case of the relevant results where $|\Omega|$, denoted $v$ in \cite{liebler2014neighbour} and \cite{neunhoffer2014sporadic}, is even. 
 
 First suppose that $\Aut(C)$ acts intransitively or imprimitively on $\Omega$. Then, by \cite[Theorem~1.1]{liebler2014neighbour}, $C$ is as in \cite[Examples 3.1, 4.1 or 4.4]{liebler2014neighbour}. All of the codes in \cite[Example~3.1]{liebler2014neighbour} either have minimum distance $1$ in the Johnson graph or consist of only a single codeword, and thus no examples occur here. In \cite[Example~4.1]{liebler2014neighbour}, all codes have minimum distance $1$ in the Johnson graph, except the code on line 1 of \cite[Table~3]{liebler2014neighbour}, which has minimum distance $1$ in $O_{k+1}$. Thus none of these examples arise. The codes in \cite[Example~4.4]{liebler2014neighbour} require that there exist positive integers $a,b$ and $k_0$ such that $2k+1=ab$ and $k=ak_0$, and also that there exists a code $C'$ in $J(b,k_0)$ that is used to construct $C$. However, since $2k+1$ and $k$ are coprime, we have that $a=1$, from which it follows that $C=C'$ producing no examples in this case. Thus we may assume that $\Aut(C)$ acts primitively on $\Omega$ and therefore, by \cite[Theorem~1.2]{liebler2014neighbour}, $\Aut(C)$ acts $2$-transitively on $\Omega$. We deal with the $2$-transitive groups in the remainder of the proof.
 
 If $\Aut(C)$ contains a sporadic $2$-transitive group $G$ then, by \cite[Theorem~1]{neunhoffer2014sporadic}, $C$, $G$ and $k$ are as in one of the lines 1--22 of \cite[Table~1]{neunhoffer2014sporadic}, where we require that $v$ in that table is equal to $2k+1$. It follows that $G$ is one of $\PSL_2(11)$ with $k=5$, $\alt_7$ with $k=7$ or $\mg_{23}$ with $k=11$. If $G=\PSL_2(11)$ and $k=5$ then $C$ has minimum distance $\delta=1$ in $O_6$. If $G=\alt_7$ and $k=7$ then $C$ is the set of all planes of $\pg_3(2)$ and, since there exists a pair of planes whose intersection is a single point, $C$ has minimum distance $\delta=3$ in $O_8$.  If $G=\mg_{23}$ and $k=11$ then $C$ is the set of endecads, that is, the set of all subsets with characteristic vector being a weight $11$ codeword in the perfect Golay code $\G_{23}$. A pair of endecads intersect in either 6 or 7 points, and hence $C$ has minimum distance $\delta=7$ in $O_{12}$. From \cite[Table~1]{neunhoffer2014sporadic} again, we have that the stabiliser of a codeword $\alpha$ is $\mg_{11}$. Since $\mg_{11}$ acts transitively on the $11$ points of $\alpha$ and transitively on the $12$ points of $\overline{\alpha}$, it follows that $C$ is $2$-neighbour-transitive.
 
 Finally we consider the infinite families of $2$-transitive groups. First, $G$ is not a symplectic group, since each $2$-transitive action of a symplectic group has even degree. Thus one of \cite[Theorems~1.3 or 1.4]{liebler2014neighbour} holds. If part (a) of \cite[Theorem~1.3]{liebler2014neighbour} holds then there exists a prime power $q$ such that $2k+1=q^3+1$ and $k=q+1$, which implies that $q^3/2=q+1$, a contradiction. Part (b) of \cite[Theorem~1.3]{liebler2014neighbour} does not occur since $G$ has even degree. Next we consider the codes under \cite[Theorem~1.4]{liebler2014neighbour} as in Table 2. The affine cases in that table can be ruled out, in the first case, since $2k+1=q^n$ and $k=q^s$ for a prime power $q$ is impossible, and, in the second case, since $2k+1\neq 16$. The first linear case in \cite[Table~2]{liebler2014neighbour} requires that there exists an even prime power $q=2^t$ such that
 \[
  2k+1=\frac{q^d-1}{q-1}\quad\text{and}\quad k=\frac{q^s-1}{q-1}.
 \]
 This implies that $2^{dt-1}-2^{st}-2^{t-1}+1=0$. Thus $t=1$ and $s=d-1$, giving the example in part (2) of the statement of this result. Note that since the stabiliser of a hyperplane acts transitively on the set of points on the hyperplane as well as the set of points in the complement of the hyperplane, the code is $2$-neighbour-transitive. In the second linear case of \cite[Table~2]{liebler2014neighbour} we have that there exists an even prime power $q_0\geq 4$ such that $2q_0+2=q_0^2$, which does not occur. Parts (a) and (b) of \cite[Theorem~1.4]{liebler2014neighbour} do not occur, as the degree of $G$ is even in each case, and part (c) is eliminated by \cite[Proposition 13]{durante2014sets}. This completes the proof.
\end{proof}

\section{Intransitive case}\label{intransitiveSection}

In this section, we address the case where $C$ is a code in $\varGamma=K(n,k)$ and $\Aut(C)$ acts intransitively on $\Omega$. In particular, this means that we may assume that there exists disjoint non-empty sets $U$ and $V$ with $\Omega=U\cup V$ such that $\Aut(C)\leq G$, where $G\cong \Sym(U)\times\Sym(V)$. Next, we define codes with respect to a $G$-invariant function, for such a $G$, $U$ and $V$.

\begin{definition}\label{def-intransCodes}
 Let $\varGamma=K(n,k)$, let $a$ and $b$ be non-negative integers such that $a+b=k$, and let $\Omega$ be the disjoint union $U\cup V$ with $|U|=u$ and $|V|=v$. For $\alpha\in V(\varGamma)$ let $\iota(\alpha)=(|\alpha\cap U|,|\alpha\cap V|)$. Define 
  \[
   C_{{\rm int}}(u,v;a,b)=\{\alpha\in V(\varGamma)\mid \iota(\alpha)=(a,b)\}.
  \]
\end{definition}

\begin{lemma}\label{lem-intransExamples}
 Let $U$, $V$ and $C=C_{{\rm int}}(u,v;a,b)$, as in Definition~\ref{def-intransCodes} where one of the following holds:
 \begin{enumerate}[{\rm (1)}]
  \item $u=1$ and $a=0$,
  \item $n=2k+1$, $u=2a$ and $v=2b+1$, or,
  \item $u=a$.
 \end{enumerate}
 Then $G=\Sym(U)\times\Sym(V)$ is a subgroup of $\Aut(C)$ and $C$ is $G$-neighbour-transitive.
\end{lemma}

\begin{proof}
 Let $g\in G$ and $\alpha\in C$. Then, since $U^g=U$ and $V^g=V$ it follows that $|\alpha^g\cap U|=|\alpha\cap U|$ and $|\alpha^g\cap V|=|\alpha\cap V|$. Hence $\iota(\alpha^g)=\iota(\alpha)=(a,b)$, which implies that $\alpha^g \in C$ and $g\in\Aut(C)$. Thus $G\leq\Aut(C)$. For any $\alpha\in C$ we have that $\alpha\cap U\in{U \choose a}$ and $\alpha\cap V\in{V \choose b}$ and, conversely, for any $\mu\in{U \choose a}$ and $\nu\in{V \choose b}$ we have that $\iota(\mu\cup\nu)=(a,b)$ and $\mu\cup \nu\in C$. Since $\Sym(U)$ acts transitively on ${U\choose a}$ and $\Sym(V)$ acts transitively on ${U\choose b}$, we deduce that $G$ acts transitively on $C$. Note that $\Sym(U)$ acts transitively on ${U\choose a}$ and $\Sym(V)$ acts transitively on ${U\choose b}$ for any $0\leq a\leq u$ and $0\leq b\leq v$, which implies that $G$ acts transitively on the set of all vertices of a given type, a fact that we shall use repeatedly in this proof.
 
 Suppose part (1) holds, that is, $u=1$ and $a=0$. This implies that $|V|=n-1$ and $b=k$. It follows that $C={V\choose k}$. If $\alpha\in C$ and $\beta\in \varGamma_1(\alpha)$ then either $\iota(\beta)=(0,k)$, in which case $\beta\in C$, or $\iota(\beta)=(1,k-1)$. Hence $C_1\subseteq\{\gamma\in V(\varGamma)\mid \iota(\gamma)=(1,k-1)\}$. Since $G$ acts transitively on the set of all vertices of type $(1,k-1)$, and $G\leq \Aut(C)$, it follows that $C_1=\{\gamma\in V(\varGamma)\mid \iota(\gamma)=(1,k-1)\}$ and $C$ is $G$-neighbour-transitive.
 
 Suppose part (2) holds. Let $\alpha\in C$. Then $|U\setminus \alpha|=a$ and $|V\setminus \alpha|=b+1$. Thus, if $\beta\in\varGamma_1(\alpha)$ then either $\beta$ has type $(a,b)$, in which case $\beta\in C$, or $\beta$ has type $(a-1,b+1)$. Hence $C_1$ is a subset of the set of all vertices of type $(a-1,b+1)$. Since $G$ acts transitively on the set of all vertices of type $(a-1,b+1)$, and $G\leq \Aut(C)$, it follows that $C_1=\{\gamma\in V(\varGamma)\mid \iota(\gamma)=(a-1,b+1)\}$ and $C$ is $G$-neighbour-transitive.
 
 Suppose part (3) holds. If $\alpha\in C$ then $|U\setminus \alpha|=0$ and $|V\setminus \alpha|=n-k$. Thus, if $\beta\in\varGamma_1(\alpha)$ then $\beta$ has type $(0,k)$ so that $C_1\subseteq\{\gamma\in V(\varGamma)\mid \iota(\gamma)=(0,k)\}$. Again, since $G$ acts transitively on the set of all vertices of type $(0,k)$, and $G\leq \Aut(C)$, it follows that $C_1=\{\gamma\in V(\varGamma)\mid \iota(\gamma)=(0,k)\}$ and $C$ is $G$-neighbour-transitive.
\end{proof}

\begin{lemma}\label{lem-intrans}
 Let $C$ be a non-trivial neighbour-transitive code in $K(n,k)$ and suppose there exist non-empty disjoint subsets $U,V\subseteq \Omega$ such that $\Omega=U\cup V$ and $\Aut(C)\leq \Sym(U)\times \Sym(V)$. In particular, $\Aut(C)$ acts intransitively on $\Omega$. Then, up to interchanging $U$ and $V$, one of the following holds.
 \begin{enumerate}[{\rm (1)}]
  \item $\delta=1$, $|U|=1$ and $\alpha\subseteq V$ for all $\alpha\in C$.
  \item $\delta=1$, $n=2k+1$ and there exists an integer $a<k$ such that $|U|=2a$ and $|\alpha\cap U|=a$ for all $\alpha\in C$.
  \item $\delta\geq 2$, $|U|<k$ and $U\subseteq\alpha$ for all $\alpha\in C$.
 \end{enumerate}
\end{lemma}

\begin{proof}
 %Since $\Aut(C)$ is intransitive on $\Omega$, . 
 For any vertex $\alpha$ of $K(n,k)$, define $\iota(\alpha)=(|\alpha\cap U|,|\alpha\cap V|)$. Then, since $U$ and $V$ are fixed setwise by $\Aut(C)$, it follows that $\iota$ is $\Aut(C)$-invariant. Hence, by Lemma~\ref{invlemma}, there exist integers $a,b$, with $0\leq a\leq |U|$ and $0\leq b\leq |V|$, such that for all $\alpha\in C$ we have $\iota(\alpha)=(a,b)$. Let $\alpha\in C$. Then, for any integers $c,d\geq 0$ satisfying $c+d=k$, $c\leq |U|-a$ and $d\leq |V|-b$, there exist $\mu\in{{U\setminus\alpha}\choose c}$ and $\nu\in{{U\setminus\alpha}\choose d}$ such that $\mu\cup\nu\in\varGamma_1(\alpha)$ and $\iota(\mu\cup\nu)=(c,d)$. Conversely, if $\beta\in\varGamma_1(\alpha)$ with $\iota(\beta)=(c,d)$ then $c,d\geq 0$, $c+d=k$, $c\leq |U|-a$ and $d\leq |V|-b$. Hence, the number of different types of vertices in $\varGamma_1(\alpha)$ is the size of the set
 \[
  S=\{(c,d)\in\Z\times \Z\mid  c,d\geq 0,\, c+d=k,\, a+c\leq |U|,\, b+d\leq |V|\}.
 \]
 Since we are assuming that $C_1$ is non-empty, we have, by Lemma~\ref{invlemma}, that either $|S|=1$, or $\delta=1$ and $|S|=2$.
 
 First, suppose that $|S|=1$. Moreover, suppose that $(0,k)\in S$. This implies that $a\leq |U|$ and $b+k\leq |V|$. Since $|S|=1$ we have that $(1,k-1)\notin S$. Hence, $a+1> |U|$ which gives $|U|=a$, and part (3) occurs. A similar argument holds, upon interchanging $U$ and $V$, if $(k,0)\in S$. Thus, we may assume that $S=\{(e,f)\}$ with $e,f>0$, in which case $a+e\leq |U|$ and $b+f\leq |V|$ hold. Since $|S|=1$, we have that $(e+1,f-1)\notin S$ and, since $f-1\geq 0$ and $b+f-1\leq|V|$ hold, it follows $a+e+1>|U|$. The two inequalities involving $|U|$ together imply that $a+e=|U|$. Also $(e-1,f+1)\notin S$, and, since $e-1\geq 0$ and $a+e-1\leq|U|$ hold, we have $b+f+1>|V|$. Thus $b+f=|V|$. However, this implies that $n=|U|+|V|=a+b+e+f=2k$, which contradicts $n\geq 2k+1$. 
 
 Assume now that $|S|=2$, and hence $\delta=1$. Then $C\cap \varGamma_1(\alpha)$ is non-empty, and hence $(a,b)\in S$, which implies that $2a\leq|U|$ and $2b\leq|V|$. Suppose that $(a,b)=(0,k)$. Then $(0,k)\in S$ and, since we are assuming that $|U|\geq 1$, we also have that $(1,k-1)$ satisfies the conditions to be in $S$. Thus $S=\{(0,k),(1,k-1)\}$. In particular, $(2,k-2)\notin S$, which implies that $2>|U|$, that is, $|U|=1$, and part (1) holds. A similar argument holds for $(a,b)=(k,0)$, and hence we may assume that $a,b>0$. Up to interchanging $U$ and $V$, we may then assume that $S=\{(a,b),(a-1,b+1)\}$. Since $(a+1,b-1)$ is not in $S$, we have that $2a+1>|U|$, and hence $|U|=2a$. Since $(a-2,b+2)$ is not in $S$, we have that $2b+2>|V|$, and hence $|V|=2b$ or $2b+1$. If $|V|=2b$ then $n=|U|+|V|=2a+2b=2k$, contradicting $n\geq 2k+1$. Thus $|V|=2b+1$ so that $n=2a+2b+1=2k+1$, and part (2) holds.
\end{proof}

In the next two lemmas we prove that the automorphism group of a neighbour-transitive code $C$ in a Kneser graph has at most two orbits on $\Omega$. The next lemma introduces the subsets $U_0$ and $U_1$ that consist of the elements of $\Omega$ contained, respectively, in no codeword or in all codewords of $C$. 

\begin{lemma}\label{lem-U0U1orbits}
 Let $C$ be a code in $K(n,k)$ and let $U_0=\Omega \setminus ( \bigcup_{\alpha\in C} \alpha )$ and $U_1=\bigcap_{\alpha\in C} \alpha$. Then $\Aut(C)$ induces $\Sym(U_0)$ on $U_0$ and $\Sym(U_1)$ on $U_1$. In particular, $U_0$ and $U_1$ are $\Aut(C)$-orbits.
\end{lemma}

\begin{proof}
 Let $\sigma\in\Sym(U_0)$. Define $x_\sigma\in\Sym(\Omega)$ by $a^{x_\sigma}=a^\sigma$ if $a\in U_0$ and $a^{x_\sigma}=a$ if $a\in \Omega\setminus U_0$. Then, for any $\alpha\in C$, we have that $\alpha\cap U_0$ is empty, and hence $\alpha^{x_\sigma}=\alpha$. Thus $x_\sigma\in\Aut(C)$. Let $\tau\in\Sym(U_1)$. Define $x_\tau\in\Sym(\Omega)$ by $a^{x_\tau}=a^\tau$ if $a\in U_1$ and $a^{x_\tau}=a$ if $a\in \Omega\setminus U_1$. Then, for any $\alpha\in C$, we have that $\alpha\cap U_1=U_1$, and hence $\alpha^{x_\tau}=\alpha$. Thus $x_\tau\in\Aut(C)$. It remains to show that $\Aut(C)$ fixes both $U_0$ and $U_1$ setwise. Let $x\in\Aut(C)$. Then, since $x$ fixes $C$ setwise, we have that $U_0^x=\Omega^x \setminus ( \bigcup_{\alpha\in C} \alpha^x )=\Omega \setminus ( \bigcup_{\alpha\in C} \alpha )=U_0$ and $U_1^x=\bigcap_{\alpha\in C} \alpha^x=\bigcap_{\alpha\in C} \alpha=U_1$. Thus the result holds.
\end{proof}

\begin{lemma}\label{intransTwoOrbits}
 Let $C$ be a neighbour-transitive code in $K(n,k)$. Then $\Aut(C)$ has at most $2$ orbits on $\Omega$.
\end{lemma}

\begin{proof}
 Suppose, to the contrary, that $W_1,W_2,W_3$ are pairwise disjoint, non-empty subsets of $\Omega$ such that $W_1\cup W_2\cup W_3=\Omega$ and $\Aut(C)\leq \Sym(W_1)\times\Sym(W_2)\times\Sym(W_3)$.
 
 First, suppose that $\delta=1$. We shall apply Lemma~\ref{lem-intrans} in three different ways: (i) on the sets $W_1$ and $W_2\cup W_3$, (ii) on the sets $W_2$ and $W_1\cup W_3$, and (iii) on the sets $W_3$ and $W_1\cup W_2$. Since $|W_i\cup W_j|\geq 2$ whenever $i$ and $j$ are distinct, if Lemma~\ref{lem-intrans}~(1) holds in each of the cases (i), (ii) and (iii), then it follows that $|W_1|=|W_2|=|W_3|=1$ and $n=3$. However, we are assuming $k\geq 2$, so that $n\geq 5$. Thus Lemma~\ref{lem-intrans}~(2) holds in at least one of the cases (i), (ii) or (iii), which implies that $n=2k+1$. Since $n$ is odd, either exactly one of $|W_1|,|W_2|,|W_3|$ is odd, or all three are. Define, for any vertex $\alpha$, the $\Aut(C)$-invariant $\iota(\alpha)=(|\alpha\cap W_1|,|\alpha\cap W_2|,|\alpha\cap W_3|)$. Note that when $n=2k+1$ and either part (1) or (2) of Lemma~\ref{lem-intrans} holds (for the respective cases (i), (ii) and (iii)), then for any $\alpha\in C$, if $|W_i|$ is even then $|\alpha\cap W_i|=|W_i|/2$, and if $|W_i|$ is odd then $|\alpha\cap W_i|=(|W_i|-1)/2$. Thus, if $|W_1|,|W_2|,|W_3|$ are all odd, we have that $k=(|W_1|+|W_2|+|W_3|-3)/2=(n-3)/2\neq (n-1)/2$. Thus, precisely one of $|W_1|,|W_2|,|W_3|$ is odd, and so we may assume that $|W_1|$ is odd. Let $a_1=(|W_1|-1)/2$, $a_2=|W_2|/2$ and $a_3=|W_3|/2$. Then, for every $\alpha\in C$ we have that $\iota(\alpha)=(a_1,a_2,a_3)$. However, there exist elements of $\varGamma_1(\alpha)$ of type $(a_1+1,a_2-1,a_3)$ and $(a_1+1,a_2,a_3-1)$, contradicting Lemma~\ref{invlemma}.
 
 Now suppose $\delta\geq 2$. Again, Lemma~\ref{lem-intrans}~(3) applies in three different ways (i) on the sets $W_1$ and $W_2\cup W_3$, (ii) on the sets $W_2$ and $W_1\cup W_3$, and (iii) on the sets $W_3$ and $W_1\cup W_2$, so that one set of each pair is a subset in every codeword of $C$. Since $\Aut(C)$ is not transitive on $W_1\cup W_2$, it follows from Lemma~\ref{lem-U0U1orbits} that $W_1\cup W_2$ is not a subset of $U_1$, as defined in Lemma~\ref{lem-U0U1orbits}. Similarly, $W_1\cup W_3$ and $W_2\cup W_3$ are not subsets of $U_1$. However, this implies that each set $W_1$, $W_2$ and $W_3$ are each contained in every codeword $\alpha\in C$, which implies that $k=n$, giving a contradiction and completing the proof.
\end{proof}

We are now in a position to prove Theorem~\ref{intransitiveParamThm}.

\begin{proof}[Proof of Theorem~\ref{intransitiveParamThm}]
 Theorem~\ref{intransitiveParamThm} follows from Lemmas~\ref{lem-intrans} and~\ref{intransTwoOrbits}, as well as the examples given in Lemma~\ref{lem-intransExamples}.
\end{proof}

To complete this section, we present an example of a code satisfying Theorem~\ref{intransitiveParamThm}~(2) that is a proper subset of the code $C_{{\rm int}}(5,8;2,4)$.

\begin{example}\label{ag32tetrahedronEx}
 Let $\Omega$ be the disjoint union $U\cup V$, where $|U|=5$ and $V$ is the set of points of the affine geometry $\ag_3(2)$. Furthermore, let $\T$ be the set of all \emph{tetrahedrons} of $V$, where a tetrahedron is set of $4$ point of $\ag_3(2)$ that do not form a $2$-flat (affine plane), and let $C$ be the code in $O_7=K(13,6)$ consisting of all vertices $\alpha$ such that $|\alpha\cap U|= 2$ and $\alpha\cap V\in\T$. 
\end{example}

\begin{lemma}
 Let $U$, $V$, $C$ and $\T$ be as in Example~\ref{ag32tetrahedronEx}. Then $C$ is a neighbour-transitive code in $O_7$ with $|C|=560$, minimum distance $\delta=1$ and automorphism group $\Aut(C)=\s_5\times \AGL_3(2)$. Moreover, $C$ satisfies Theorem~\ref{intransitiveParamThm}~(2).
\end{lemma}

\begin{proof}
 There are ${8\choose 4}=70$ subsets of $V$ having size $4$, of which precisely $14$ are $2$-flats of $\ag_3(2)$. Hence $|\T|=56$. Moreover, there are ten $2$-subsets of $U$, which implies that $|C|=560$. We now determine $\Aut(C)$. The stabiliser of $U$ in $\Sym(\Omega)$ is $\s_5\times\s_8$ with the subgroup $\s_5$ in the first component acting transitively on ${U\choose 2}$. Acting on $V$, the stabiliser of $\T$ inside $\s_8$ is the stabiliser of the set of all $2$-flats, that is, $\AGL_3(2)$. Hence $\Aut(C)=\s_5\times \AGL_3(2)$. 
 
 Next we prove that $C$ is neighbour-transitive. If $e_1,e_2,e_3$ is a basis for the underlying vector space of $\ag_3(2)$ then every element of $\T$ is equivalent under $\AGL_3(2)$ to $\{0,e_1,e_2,e_3\}$. Since $\s_5$ acts transitively on ${U\choose 2}$ we have that $\Aut(C)$ acts transitively on $C$. Note that the complement, inside $V$, of a tetrahedron is again a tetrahedron, and hence if $\alpha\in C$ then $|\overline{\alpha}\cap U|=3$ and $\overline{\alpha}\cap V\in \T$. This implies that for every $\alpha\in C$ and $\beta\in\varGamma_1(\alpha)$ we have that either $\beta\in C$, or $|\beta\cap U|=3$ and $\beta\cap V$ is a triangle in $\ag_3(2)$. Thus $C_1$ is precisely the set of all vertices $\beta$ such that $|\beta\cap U|=3$ and $\beta\cap V$ is a triangle in $\ag_3(2)$. Since $\s_5$ acts transitively on ${U\choose 3}$ and $\AGL_3(2)$ acts transitively on triangles, it follows that $\Aut(C)$ acts transitively on $C_1$. Hence $C$ is neighbour-transitive. Finally, since $|\alpha\cap V|=|V|/2$ for all $\alpha\in C$, we have that $C$ satisfies part (2) of Theorem~\ref{intransitiveParamThm}.
\end{proof}

\section{Imprimitive case}\label{imprimitiveSection}

In this section, we consider codes $C$ in $\varGamma=K(2k+1,k)=O_{k+1}$ where the action of $\Aut(C)$ is imprimitive on $\Omega$. In this case, $\Aut(C)\leq G$ where $G\cong \s_a\wr\s_b$ is the stabiliser inside $\Sym(\Omega)$ of a partition $\B=\{B_1,\ldots,B_a\}$ of $\Omega$ into $a$ blocks each of size $b$, with $a,b\geq 2$. For any subset $\alpha\subseteq \Omega$, we define $\iota(\alpha)$ to be the multiset $\{|\alpha\cap B_1|,\ldots,|\alpha\cap B_a|\}$. Given a multiset $S$ we sometimes write $S=\{b_1^{a_1},\ldots,b_s^{a_s}\}$ where each $b_i$ is an element of $S$ that occurs with multiplicity $a_i$, for $i=1,\ldots,s$. For example, the multiset $\{0,1,1,2,2,2\}$ could be written as $\{0^1,1^2,2^3\}$. Hence, if $\alpha\subseteq \Omega$ with $\iota(\alpha)=\{b_1^{a_1},\ldots,b_s^{a_s}\}$ then, since $a$ is the number of blocks in $\B$, we have that $a_1+\cdots+a_s=a$. Moreover, $a_1 b_1 + \cdots + a_s b_s = |\alpha|$ and $\iota(\overline{\alpha})=\{(b-b_1)^{a_1},\ldots,(b-b_s)^{a_s}\}$, where $\overline{\alpha}$ is the complement of $\alpha$ in $\Omega$.

%\textcolor{blue}{DH -- We need to add examples for the cases from Theorem~\ref{imprimitiveOddThm} here. Again, there's the obvious examples (take all vertices $\alpha$ with required $\iota(\alpha)$), and it would be good to include one or two interesting examples from the computations Nina did (if we can find a nice description for some).}

 \begin{definition}\label{def-imprimitiveExamples}
  Let $\varGamma=K(2k+1,k)=O_{k+1}$ and let $\B=\{B_1,\ldots,B_a\}$ be a partition of  $\Omega$ where $|B_i|=b$ for each $i=1,\ldots,a$. For a vertex $\alpha$ let $\iota(\alpha)$ be the multiset $ \{\alpha\cap B_1,\ldots,\alpha\cap B_a\}$. For a multiset $M$, define
  \[
   C_{{\rm imp}}(a,b;M)=\{\alpha\in V(\varGamma)\mid \iota(\alpha)=M\}.
  \]
 \end{definition}

\begin{lemma}\label{lem-imprimitiveExamples}
 Let $a,b,\B$ and $C=C_{{\rm imp}}(a,b;M)$ be as in Definition~\ref{def-imprimitiveExamples}, so that $C$ is a code in $\varGamma=O_{k+1}=K(2k+1,k)$, and suppose one of the following holds:
 \begin{enumerate}[{\rm (1)}]
  \item $M=\left\{((b-1)/2)^{(a+1)/2},((b+1)/2)^{(a-1)/2}\right\}$,
  \item $M=\left\{0^{(a-1)/2},(b-1)/2,b^{(a-1)/2}\right\}$, or,
  \item $M=\left\{b^{a_0},b_1^{a_1}\right\}$, for some integers $a_0,a_1,b_1\geq 0$, where $a_0 b+a_1 b_1=k$ and $(a_0+a_1)b=2k+1$.% a_1 b-a_1 b_1=k+1
 \end{enumerate}
 Then $C$ is $G$-neighbour-transitive, where $G\cong \s_b\wr\s_a$ is the largest subgroup of $\Sym(\Omega)$ preserving the partition $\B$.
\end{lemma}
 
\begin{proof}
  Since $G$ preserves the partition $\B$, and $\iota$ is defined in terms of $\B$, it follows that $\iota(\alpha^g)=\iota(\alpha)$ for all $\alpha\in V(\varGamma)$ and $g\in G$. Moreover, since $G$ is $a$-transitive on the set $\{B_1,\ldots,B_a\}$ of blocks and independently $b$-transitive on each block $B_i$, for $i=1,\ldots,a$, it follows that $G$ acts transitively on the set of all vertices of a fixed type. In particular, $G$ is a subgroup of $\Aut(C)$ and $G$ acts transitively on $C$.
  
  Suppose case (1) holds, that is, $M=\{((b-1)/2)^{(a+1)/2},((b+1)/2)^{(a-1)/2}\}$. Let $\alpha\in C$. Then the complement $\overline{\alpha}$ has type $\{((b-1)/2)^{(a-1)/2},((b+1)/2)^{(a+1)/2}\}$. If $\beta\in\varGamma_1(\alpha)$ then $\beta$ is obtained from $\overline{\alpha}$ by removing one of its elements. Thus, either $\beta$ has type $\{((b-1)/2)^{(a+1)/2},((b+1)/2)^{(a-1)/2}\}$ and is in $C$, or $\beta$ has type $\{(b-3)/2,((b-1)/2)^{(a-3)/2},((b+1)/2)^{(a+1)/2}\}$ and is in $C_1$. Since $G$ acts transitively on the set of all vertices of a given type, and $G\leq\Aut(C)$, it follows that $C_1$ is the set of all vertices of type $\{(b-3)/2,((b-1)/2)^{(a-3)/2},((b+1)/2)^{(a+1)/2}\}$ and $C$ is $G$-neighbour-transitive.
  
  Suppose case (2) holds. Let $\alpha\in C$. Then the complement $\overline{\alpha}$ has type $\{0^{(a-1)/2},(b+1)/2,b^{(a-1)/2}\}$. Then $\beta\in\varGamma_1(\alpha)$ is obtained from $\overline{\alpha}$ by removing one of its elements, and hence $\beta$ either has type $\{0^{(a-1)/2},(b-1)/2,b^{(a-1)/2}\}$ and is in $C$, or has type $\{0^{(a-1)/2},(b+1)/2,b-1,b^{(a-3)/2}\}$ and is in $C_1$. Since $G$ acts transitively on the set of all vertices of a given type, and $G\leq\Aut(C)$, it follows that $C_1$ is the set of all vertices of type $\{0^{(a-1)/2},(b+1)/2,b-1,b^{(a-3)/2}\}$ and $C$ is $G$-neighbour-transitive.
  
  Suppose case (3) holds. Here, if $\alpha\in C$ and $\beta\in\varGamma_1(\alpha)$ then $\alpha$ has type $\{b^{a_0},b_1^{a_1}\}$, $\overline{\alpha}$ has type $\{0^{a_0},(b-b_1)^{a_1}\}$ and hence $\beta$ has type $\{0^{a_0},b-b_1-1,(b-b_1)^{a_1-1}\}\neq M$, which implies that $\beta\in C_1$. Again, since $G$ acts transitively on the set of all vertices of a given type, and $G\leq\Aut(C)$, it follows that $C_1$ is the set of all vertices of type $\{0^{a_0},b-b_1-1,(b-b_1)^{a_1-1}\}$ and $C$ is $G$-neighbour-transitive.
\end{proof}

%%%%%%%%%%%%%%%%%%%%%%%%%%%%%%%%%%%%%%%%%%%%%%%%%%%%%%%%%%

\begin{lemma}\label{imprimitiveOddLemma1}
 Let $C$ be a non-trivial neighbour-transitive code in $K(2k+1,k)=O_{k+1}$ with minimum distance $\delta\geq2$ and suppose that $\Aut(C)$ acts transitively but imprimitively on $\Omega$ with system of imprimitivity $\{B_1,\ldots,B_a\}$, where $|B_i|=b$ for each $i$ and $ab=2k+1$ with $a,b\geq 2$. Moreover, for each $\alpha\subseteq \Omega$ let $\iota(\alpha)$ be the multiset $\{|\alpha\cap B_i|\mid i=1,\ldots,a\}$. Then there exist integers $a_0,a_1$ and $b_1$ such that $\iota(\alpha)=\{b^{a_0},b_1^{a_1}\}$ for all $\alpha\in C$.
\end{lemma} 

\begin{proof}
	For $\alpha\in C$ and $\beta\in \Gamma_1(\alpha)$, we have that $\beta=\overline{\alpha}\setminus\{u\}$, where $\overline{\alpha}$ is the complement of $\alpha$ in $\Omega$ and $u$ is some element from $\overline{\alpha}$. 
	If $\iota(\alpha)=\{b_1^{a_1},\ldots , b_s^{a_s}\}$, then
	$\iota(\overline{\alpha})$ is of the form 
	$\{(b-b_1)^{a_1},\ldots , (b-b_s)^{a_s}\}$, so $\iota(\beta)= \{ (b-b_1)^{a_1},\ldots , (b-b_i)^{a_i-1},(b-b_i-1), \ldots , (b-b_s)^{a_s} \}$, for some $i\in\{1,\ldots ,s\}$.  This implies that there are $s$ different types for a neighbour $\beta$ of $\alpha$, unless $b=b_j$ for some $j\in\{1,\ldots ,s\}$, in which case there are $s-1$ types. 
	Since Lemma \ref{invlemma} ensures that all elements of $\Gamma_1(\alpha)$ have the same type, it follows that $s=1$ or $s-1=1$ (i.e. $s=2$). Therefore $\iota(\overline{\alpha})=\{0^{a_0},(b-b_1)^{a_1}\}$, with $a_0\geq 0$. Then $\iota(\alpha)$ is either equal to $\{b^{a_0},b_1^{a_1}\}$ or to $\{b_1^{a}\}$. 
	Suppose $\iota(\alpha)= \{b_1^{a}\}$. Then $\iota(\overline{\alpha})=\{(b-b_1)^a \}$. This implies $ab_1=|\alpha|=k$ and $a(b-b_1)=|\overline{\alpha}|=k+1$. Hence $a$ divides both $k$ and $k+1$, that is $a=1$, a contradiction.
	We conclude that $\iota(\alpha)=\{b^{a_0},b_1^{a_1}\}$.
\end{proof}

%%%%%%%%%

\begin{lemma}\label{imprimitiveOddLemma2}
 Let $C$ be a non-trivial neighbour-transitive code in $K(2k+1,k)=O_{k+1}$ with minimum distance $\delta=1$ and suppose that $\Aut(C)$ acts transitively but imprimitively on $\Omega$ with system of imprimitivity $\{B_1,\ldots,B_a\}$, where $|B_i|=b$ for each $i$ and $ab=2k+1$ with $a,b\geq 2$. Moreover, for each $\alpha\subseteq \Omega$ let $\iota(\alpha)$ be the multiset $\{|\alpha\cap B_i|\mid i=1,\ldots,a\}$. Then one of the following holds.
 \begin{enumerate}[{\rm (1)}]
  \item $\iota(\alpha)=\left\{\left(\frac{b-1}{2}\right)^{(a+1)/2},\left(\frac{b+1}{2}\right)^{(a-1)/2}\right\}$ for all $\alpha\in C$.
  \item $\iota(\alpha)=\left\{0^{(a-1)/2},\frac{b-1}{2},b^{(a-1)/2}\right\}$ for all $\alpha\in C$.
 \end{enumerate}
\end{lemma} 

\begin{proof}
	If $\delta=1$, then for $\alpha\in C$ there exists $\beta\in\Gamma_1(\alpha)$ such that $\iota(\alpha)=\iota(\beta)$, i.e. a neighbour $\beta$ of $\alpha$ such that $\beta \in C$. 

	Suppose that $\iota(\beta)=\iota(\alpha)$, for each $\beta\in\Gamma_1(\alpha)$. 
	%%%%%%%%
	If $\iota(\alpha)=\{b_1^{a_1},\ldots , b_s^{a_s}\}$, then the complement $\overline{\alpha}$ has type
	$\{(b-b_1)^{a_1},\ldots , (b-b_s)^{a_s}\}$.  Element $\beta\in\varGamma_1(\alpha)$ is obtained from $\overline{\alpha}$ by removing one of its elements, and hence $\beta$ has type $\iota(\beta)= \{(b-b_1)^{a_1},\ldots , (b-b_i)^{a_i-1},(b-b_i-1), \ldots , (b-b_s)^{a_s} \}$, for some $i\in\{1,\ldots ,s\}$. This implies that there are $s$ different types for a neighbour $\beta$ of $\alpha$, unless $b=b_j$ for some $j\in\{1,\ldots ,s\}$, in which case there are $s-1$ types. Since we assumed that all neighbours of $\alpha$ are of the same type as $\alpha$, it follows that $s=1$ or $s=2$ and
	$\iota(\overline{\alpha})= \{(b-b_1)^{a_1},0^{a_2}\}$.  Therefore $\iota(\beta)= \{ (b-b_1)^{a_1-1},(b-b_1-1) ,0^{a_2} \} =\{ b_1^{a_1}, 0^{a_2}\}$. Equality of the last two multisets implies that $a_1=1$ and $b-b_1-1=b_1$, so $b_1=(b-1)/2$. Moreover, then it holds $b_1=a_1b_1=k$. In this case we obtain $b=2k+1$, a contradiction.
	
	It follows that there exist $\beta_1,\beta_2\in\Gamma_1(\alpha)$ such that $\iota(\beta_1)=\iota(\alpha)\neq\iota(\beta_2)$.
	%%%%%%%%%
	As before, if $\iota(\alpha)=\{b_1^{a_1},\ldots , b_s^{a_s}\}$, then 
	$\iota(\overline{\alpha}) =\{(b-b_1)^{a_1},\ldots , (b-b_s)^{a_s}\}$, so any $\beta\in\varGamma_1(\alpha)$ has type $\iota(\beta)= \{(b-b_1)^{a_1},\ldots , (b-b_i)^{a_i-1},(b-b_i-1), \ldots , (b-b_s)^{a_s} \}$, for some $i\in\{1,\ldots ,s\}$. Then there are $s$ different types for a neighbour $\beta$ of $\alpha$, unless $b=b_j$ for some $j\in\{1,\ldots ,s\}$, in which case there are $s-1$ types. Since by Lemma \ref{invlemma} 
	every vertex from $\varGamma_1(\alpha)$ is the same type as $\alpha$ or $\beta_2$, it follows that $s=2$ or $s=3$.	%%%%%%%%
	Therefore, we obtain $\iota(\overline{\alpha})=\{ (b-b_1)^{a_1},(b-b_2)^{a_2}, 0^{a_3} \}$. Depending on the value of $a_3$ we have the following two cases.
	\begin{itemize}
		\item[i)] If $a_3=0$ we have $\iota(\alpha)=\{b_1^{a_1}, b_2^{a_2}\}$, so $\iota(\beta)=\{ (b-b_1)^{a_1-1}, b-b_1-1, (b-b_2)^{a_2} \}$ or $\iota(\beta)=\{ (b-b_1)^{a_1}, (b-b_2)^{a_2-1}, b-b_2-1 \}$, for $\beta\in\Gamma_1(\alpha)$.
 		We can assume that:
		$$\iota(\beta_1)=\{ (b-b_1)^{a_1-1}, b-b_1-1, (b-b_2)^{a_2} \}, \ \iota(\beta_2)= \{ (b-b_1)^{a_1}, (b-b_2)^{a_2-1}, b-b_2-1 \}.$$		
		Note that $b_2\neq b-b_2$, since otherwise $2b_2=b$, but $b$ must be odd. Hence, $b-b_2=b_1$, because $\iota(\beta_1)=\iota(\alpha)$. Also, $b-b_1-1=b-b_2$ so $b_2=b_1+1$. This tells us that $a_2+1=a_1$. Since $a_1+a_2=a$ and $b_1+b_2=b$ in this case, we obtain: $a_1=(a+1)/2$, $a_2=(a-1)/2$, $b_1=(b-1)/2$, $b_2=(b+1)/2$. Therefore, $\iota(\alpha)= \{ \left( \frac{b-1}{2}\right)^{\frac{a+1}{2}}, \left(  \frac{b+1}{2}\right)^{\frac{a-1}{2}} \}$.		
		
		\item[ii)] Let $a_3\geq1$. We can assume
		$\iota(\beta_1)=\{ (b-b_1)^{a_1-1}, b-b_1-1, (b-b_2)^{a_2}, 0^{a_3} \}$ and $\iota(\beta_2)= \{ (b-b_1)^{a_1}, (b-b_2)^{a_2-1}, b-b_2-1, 0^{a_3} \}$.	We have $\iota(\alpha)=\{ b_1^{a_1},b_2^{a_2},b^{a_3} \}$ and $\iota(\alpha)=\iota(\beta_1)$. First, $b_2\neq b-b_2$, as before. Either $b_1$ or $b_2$ must be equal to $0$.
	
	If $b_1=0$, then $\iota(\alpha)=\{ 0^{a_1},b_2^{a_2},b^{a_3} \}$ and $\iota(\beta_1)=\{ b^{a_1-1}, b-1, (b-b_2)^{a_2}, 0^{a_3} \}$. This implies that $a_1-1=a_3$ and $b-1=b-b_2=b_2$, a contradiction. 

	If $b_2=0$, then $\iota(\alpha)=\{ 0^{a_2},b_1^{a_1},b^{a_3} \}$ and $\iota(\beta_1)=\{ (b-b_1)^{a_1-1}, b-b_1-1, b^{a_2}, 0^{a_3} \}$. It follows that $a_2=a_3$ and $b_1=b-b_1-1$, so $b_1=(b-1)/2$. If $a_1\geq2$, then $b_1=b-b_1=b-b_1-1$, a contradiction. Therefore, $a_1=1$. From $a_1+a_2+a_3=a$, we obtain $a_2=a_3=(a-1)/2$, hence $\iota(\alpha)= \{ 0^{(a-1)/2},(b-1)/2,b^{(a-1)/2} \}$. 
\end{itemize}
This completes the proof.
\end{proof}

We may now prove Theorem~\ref{imprimitiveOddThm}.

\begin{proof}
 Theorem~\ref{imprimitiveOddThm} is proved by Lemma~\ref{imprimitiveOddLemma1}, which considers the case where $\delta\geq 2$, and Lemma~\ref{imprimitiveOddLemma2}, which deals with the case where $\delta=1$, as well as the examples provided in Lemma~\ref{lem-imprimitiveExamples}.
\end{proof}

\section{Transitive action on \texorpdfstring{$\Omega$}{Omega}}\label{transitiveSection}

%\begin{remark}\label{deltageq3rem}
% Note that if $C$ is a code in $K(n,k)$ with minimum distance $\delta\geq 3$ then the diameter of $K(n,k)$ must be at least $3$
%\end{remark}

We first give an example. Note that the example below has minimum distance $1$, but has a primitive (and hence transitive) group acting on the underlying set $\Omega$ (see Lemma~\ref{lem-primitiveExample}).

\begin{example}\label{halfLinesinPlaneEx}
 Let $V=\langle e_1,e_2,e_3\rangle\cong\F_3^3$ be the underlying vector space of $\pg_2(3)$. Let $\varGamma=O_7=K(13,6)$ with underlying set $\Omega$ being the point set of $\pg_2(3)$. Define the code $C$, in $\varGamma$, to be the set of all vertices $\alpha$ such that $\alpha$ is the symmetric difference $\ell_1\ominus\ell_2$ for some distinct pair of lines $\ell_1,\ell_2$ of $\pg_2(3)$.
\end{example}

\begin{lemma}\label{lem-primitiveExample}
 Let $C$ be as in Example~\ref{halfLinesinPlaneEx}. Then $C$ is neighbour-transitive, with $|C|=78$ minimum distance $\delta=1$ and $\Aut(C)=\PGaL_3(3)$.
\end{lemma}

\begin{proof}
 Since $\PGaL_3(3)$ is the full collineation group of $\pg_2(3)$ it follows that $\Aut(C)=\PGaL_3(3)$ and, since $\PGaL_3(3)$ acts $2$-transitively on lines, we have that $\Aut(C)$ acts transitively on $C$. There are ${13\choose 2}=78$ pairs of lines, so that $|C|=78$. There are four lines of $\pg_2(3)$ through any given point $p$, say $\ell_1,\ell_2,\ell_3,\ell_4$, and $\ell_1\ominus\ell_2$ is disjoint from $\ell_3\ominus\ell_4$, both of which are in $C$. Hence $\delta=1$. If $\alpha=\ell_1\ominus\ell_2$ then $\varGamma_1(\alpha)$ contains $\ell_3\ominus\ell_4$ and $(\ell_3\cup\ell_4)\setminus p_1$ for some $p_1\in\ell_3\ominus\ell_4$, the latter being in $C_1$. The stabiliser in $\PGL_3(2)$ of $\ell_3\cap\ell_4$ acts transitively on the points of $\ell_3\ominus\ell_4$, and hence transitively on $\varGamma_1(\alpha)\cap C_1$. Thus $C$ is neighbour-transitive, completing the proof.
\end{proof}

\begin{lemma}\label{homogeneouslemma}
 Let $C$ be a neighbour-transitive code in $K(n,k)$ with minimum distance $\delta\geq 3$ and let $\alpha\in C$. Then $\Aut(C)_\alpha$ acts transitively on $\varGamma_1(\alpha)$ and acts $(n-2k)$-homogeneously on $\Omega\setminus\alpha$. 
\end{lemma}

\begin{proof}
 Since $\delta\geq 3$, the set $\{\varGamma_1(\beta)\mid \beta\in C\}$ partitions $C_1$. Moreover, the sets $\varGamma_1(\alpha)$ and $\varGamma_1(\beta)$ are disjoint for distinct $\beta\in C\setminus\{\alpha\}$. As $\Aut(C)$ acts transitively on $C_1$, it follows that for all $\gamma_1,\gamma_2\in\varGamma_1(\alpha)$ there exists a $g\in\Aut(C)$ such that $\gamma_1^g=\gamma_2$. Since $\alpha$ is the unique codeword at distance $1$ from each of $\gamma_1,\gamma_2$, we have that $g\in\Aut(C)_\alpha$. Hence $\Aut(C)_\alpha$ acts transitively on $\varGamma_1(\alpha)$. Now, the neighbours of $\alpha$ are precisely the $k$-subsets of $\Omega$ that are disjoint from $\alpha$, that is,
 \[
  \varGamma_1(\alpha)={\Omega\setminus\alpha \choose k} .
 \]
 Thus, $\Aut(C)_\alpha$ is $k$-homogeneous on $\Omega\setminus\alpha$, which, upon taking complements in $\Omega\setminus\alpha$, is the same as $(n-2k)$-homogeneous, since $|\Omega\setminus\alpha|=n-k$.
\end{proof}

\begin{lemma}\label{allpairsincomplement}
 Let $C$ be a code in $K(n,k)$ and suppose that $\Aut(C)$ acts transitively on $\Omega$. Then for every pair $u,v\in\Omega$ there exists a codeword $\beta\in C$ such that $u,v\in\overline{\beta}$.
\end{lemma}

\begin{proof}
 Define a graph $\Phi$ with vertex set $\Omega$ where distinct $u,v\in\Omega$ are adjacent if there exists an $\beta\in C$ such that $u,v\in\overline{\beta}$. It follows that $\Phi$ is $\Aut(C)$-vertex-transitive. Hence the clique-coclique bound 
 \[
  \alpha(\Phi)\omega(\Phi)\leq n
 \]
 holds. Since the induced subgraph $\Phi[\overline{\beta}]$ is isomorphic to $K_{n-k}$, and $n-k\geq (n+1)/2$, we have that $\alpha(\Phi)=1$. Hence $\Phi$ is the complete graph, and the result holds.
\end{proof}

\begin{proof}[Proof of Theorem~\ref{notodd2hom}]
 Suppose that $n\neq 2k+1$, that is, that $K(n,k)$ is not an odd graph. Let $u_1,v_1,u_2,v_2\in\Omega$ with $u_1\neq v_1$ and $u_2\neq v_2$. Then, by Lemma~\ref{allpairsincomplement}, there exist $\alpha_1,\alpha_2\in C$ such that $u_1,v_1\in\overline{\alpha_1}$ and $u_2,v_2\in\overline{\alpha_2}$. Moreover, since $|\overline{\alpha_1}|=|\overline{\alpha_2}|=n-k$ and $k\leq n/2-1$, it follows that $|\overline{\alpha_1}|\cap|\overline{\alpha_2}|\geq 2$. Hence, there exist $u_3,v_3\in \overline{\alpha_1}\cap\overline{\alpha_2}$. Furthermore, by Lemma~\ref{homogeneouslemma}, there exists a $g_1\in\Aut(C)_{\alpha_1}$ and a $g_2\in\Aut(C)_{\alpha_2}$ such that $\{u_1,v_1\}^{g_1}=\{u_3,v_3\}$ and $\{u_2,v_2\}^{g_2}=\{u_3,v_3\}$. Hence $\{u_1,v_1\}^{g_1 g_2^{-1}}= \{u_2,v_2\}$ and the result follows.
\end{proof}

\end{document}